\documentclass[11pt]{amsart}

\usepackage[english]{babel}
\usepackage{geometry}
\usepackage[T1]{fontenc}
\usepackage[latin1]{inputenc}
\usepackage{amsmath,amssymb,mathrsfs,amsthm, tikz-cd,mathrsfs}

\usepackage{enumitem}

\usepackage{hyperref} 
\hypersetup{
    colorlinks=true,
    linkcolor=blue,
    urlcolor=red,
    pdftitle={},
    }

\newtheorem{theorem}{Theorem}[section]
\newtheorem{proposition}[theorem]{Proposition}

\newtheorem{lemma}[theorem]{Lemma}
\newtheorem{corollary}[theorem]{Corollary}

\newtheorem{remark}[theorem]{Remark}

\newcommand{\C}{\mathbb{C}}
\newcommand{\Z}{\mathbb{Z}}

\newcommand{\A}{\mathcal{A}}
\newcommand{\R}{\mathbb{R}}
\newcommand{\eps}{\varepsilon}
\newcommand{\X}{\mathcal{X}}
\newcommand{\Q}{\mathbb{Q}}
\newcommand{\N}{\mathbb{N}}
\newcommand{\LL}{\mathcal{L}}

\begin{document}
\title{Theta invariants and Lattice-Point Counting in Normed $\Z$-Modules}

\author{Mounir Hajli}
\thanks{ }

\address{School of Science, Westlake University}

\email{hajli@westlake.edu.cn}

            \maketitle 
            
\begin{abstract}

 Euclidean lattices occupy a central position in number theory, the geometry of numbers, and modern cryptography. In the present article, the theory of Euclidean lattices is employed to investigate normed $\mathbb{Z}$-modules of finite rank. Specifically, let $\overline{E}$ be a normed $\mathbb Z$-module of finite rank. We establish several  inequalities for the lattice-point counting function of $\overline{E}$, along with related results. Our arguments rely primarily on the analytic properties of the theta series associated with Euclidean lattices.\\

\end{abstract}

{\small{
\begin{center}
MSC: 14G40  \\

Keywords: Euclidean lattice; Theta invariants; Normed $\Z$-modules. \end{center}
}}

\tableofcontents

\section{Introduction}

 Euclidean lattices are fundamental objects in number theory and the geometry of numbers, as extensively studied in works such as \cite{Bana, BostTheta, Conway, Gaudron2013, Grayson1984, Hermite1850, Stuhler1976}. In recent decades, they have also gained significant attention in cryptography, see  \cite{LLL, MG2002, MR2009}.
Statements that relate the geometric invariants of a Euclidean lattice and its dual lattice are traditionally referred to as \emph{transference theorems}. In his seminal work \cite{Bana}, Banaszczyk established remarkable transference inequalities involving the successive minima and the covering radius of Euclidean lattices. His  approach relies on the analytical properties of the theta  series \(\theta_{\overline{E}}\) associated with Euclidean lattices \(\overline{E}\), as well as on the Poisson summation formula (see Sections \ref{sec2} and \ref{tt} for detailed definitions of these concepts). 
The functions $\theta_{\overline{E}}$  play a central role in the refined analysis of general Euclidean lattices, as demonstrated by Banaszczyk's method.

\

Let \(\overline{E} = (E, \|\cdot\|)\) be a normed \(\mathbb{Z}\)-module of finite rank \(n\). We define \(\hat{h}^0(\overline{E})\) as the real number given by
\[
\hat{h}^0(\overline{E}) = \log \# \{ s \in E \mid \|s\| \leq 1 \}.
\]
This invariant, \(\hat{h}^0(\overline{E})\), plays a pivotal role in Arakelov geometry. It serves as an arithmetic analogue to the dimension of the space of sections of vector bundles over algebraic curves. 
 We further define \(\hat{h}^1(\overline{E}) := \hat{h}^0(\overline{E}^\vee)\), where \(\overline{E}^\vee\) denotes the dual of the normed \(\mathbb{Z}\)-module \(\overline{E}\). Additionally, we consider the arithmetic degree of \(\overline{E}\), denoted by \(\widehat{\deg}(\overline{E})\); for more details on these invariants, see Section \ref{sec2}.

Gillet and Soul\'e \cite{Soule-lattice} established an arithmetic analogue of the geometric Riemann-Roch theorem for curves. This can be formulated as follows:
\begin{equation}\label{ARR}
-\log(6) \cdot n \leq \hat{h}^0(\overline{E}) - \hat{h}^1(\overline{E}) - \widehat{\deg}(\overline{E}) \leq \log\left(\frac{3}{2}\right) \cdot n + 2\log n!.
\end{equation}

 Consequently, they demonstrated that the number of lattice points in a symmetric convex body is essentially determined by its successive minima, modulo a function that depends only on the rank of the convex body. This result has significant applications in Arakelov geometry and number theory. Henk \cite{Henk} presented a remarkably simple proof of a result due to  Gillet and Soul\'e, which relates the number of lattice points in a symmetric convex body to its successive minima.\\

 In this paper, we present an alternative approach to studying lattice points.
 We establish several  inequalities for the lattice-point counting function of $\overline{E}$, along with related results.
Our bounds are somewhat coarser compared to those given by the Gillet-Soul\'e theorem. This is primarily because the  proof  of Gillet and Soul\'e employs a difficult result due to Bourgain and Milman \cite{Bourgain-Milman} which gives a sharp lower bound for the product of the volumes of the unit balls associated with a normed $\Z$-module and its dual.   Nevertheless, as noted by Boucksom in \cite{Boucksom-lattice}, these coarser  bounds remain sufficient for various applications. \\

Our approach is rooted in the theory of Euclidean lattices, more particularly on 
theta series associated with Euclidean lattices. We define \({h}_\theta^0(\overline{E})\) as the real number given by:
\[
h_\theta^0(\overline{E}) = \log \sum_{v \in E} e^{-\pi \|v\|^2}.
\]
We call it \emph{the theta invariant} of $\overline E$.
We define ${h}_\theta^1(\overline{E}) := {h}_\theta^0(\overline{E}^\vee)$, where $\overline{E}^\vee$ denotes the dual of the Euclidean lattice $\overline{E}$ (see  Section \ref{tt} for more details about these invariants).

A pivotal result in the theory of Euclidean lattices asserts that $h_\theta^0(\overline{E})$ and $\hat{h}^0(\overline{E})$ are essentially equal. More precisely, the difference $h_\theta^0(\overline{E}) - \hat{h}^0(\overline{E})=O(n\log n)$ depends only on  $n$, the rank of  ${E}$ (see Proposition \ref{strictineq1}). This result plays a crucial role in this paper. \\

Next, we exhibit briefly a class of normed \(\mathbb{Z}\)-modules and Euclidean lattices that arise naturally in Arakelov geometry. For an introduction to Arakelov geometry, see \cite{Soule}. 
Let $\X$ be a projective, integral, and flat scheme of dimension $n+1$ over $\mathbb{Z}$. Such schemes are referred to as arithmetic varieties over $\mathbb{Z}$. Assume $\X$ is an arithmetic variety over $\mathbb{Z}$ of dimension $n+1$. We assume that the generic fibre $\X_\Q$ is smooth. Let $\overline{\LL} = (\LL, \|\cdot\|_{\overline{\LL}})$ be a smooth Hermitian line bundle on $\X$. 

For any $k \in \mathbb{N}$, we write $k\overline{\LL} := {\overline{\LL}}^{\otimes k}$, and denote by $n_k$ the rank of $H^0(\X, k\LL)$. We set $X := \X(\mathbb{C})$ and $L := \mathcal{L}(\mathbb{C})$. Let $\mu$ be a smooth volume form on $\X$. The space of global sections $H^0(X, L)$ is equipped with the $L^2$-norm:
\[
\|s\|_{L^2, \overline{\LL}}^2 := \int_X \|s(x)\|_{\overline{\LL}}^2 \mu \quad \text{for any } s \in H^0(X, L).
\]
Additionally, we consider the supremum norm, defined as:
\[
\|s\|_{\sup, \overline{\LL}} := \sup_{x \in X} \|s(x)\|_{\overline{\LL}} \quad \text{for any } s \in H^0(X, L).
\]

Thus, we obtain 
two normed $\mathbb{Z}$-modules: $\overline{H^0(\X, \LL)}_{L^2, \overline{\LL}} = (H^0(\X, \LL), \|\cdot\|_{L^2, \overline{\LL}})$, which is Euclidean, and $\overline{H^0(\X, \LL)}_{\sup, \overline{\LL}} = (H^0(\X, \LL), \|\cdot\|_{\sup, \overline{\LL}})$. 
Elements of \(\overline{H^0(\X, \LL)}_{\sup, \overline{\LL}}\) with norm less than or equal to \(1\) are called  \emph{small sections}.

 The arithmetic volume $\widehat{\mathrm{vol}}(\overline{\LL})$ for a Hermitian line bundle $\overline{\LL}$ on an arithmetic variety is a fundamental invariant in Arakelov geometry. It was introduced by Moriwaki in \cite{Moriwaki1} as an analogue of the geometric volume function. Roughly speaking, this invariant measures the growth of the number of small sections of $k\overline{\LL}$ as $k \to \infty$. Yuan \cite{YuanInventiones} studied the bigness property of Hermitian line bundles on arithmetic varieties. 
The  arithmetic volume function of $\overline \LL$ is defined as follows:
\[
\widehat{\mathrm{vol}}(\overline{{\mathcal{L}}}):=
\underset{k\rightarrow \infty}{\limsup} \frac{\hat{h}^0(\overline{H^0(\X,k\LL)}_{(\sup,k\overline \LL)})
}{k^{n+1}/(n+1)!}.\]

\

The following theorem is our main result. We shall show that it is essentially a consequence of the theory of Euclidean lattices.

\begin{theorem}\label{main}
Let $\overline{E}$ be a normed $\mathbb{Z}$-module of rank $n$. Then the following inequality holds:
\[
\begin{split}
- n \log n + \log\left(1 - \tfrac{1}{2\pi}\right) - \pi 
\leq \widehat{h}^0(\overline{E}) - \widehat{h}^1(\overline{E}) - \widehat{\deg}(\overline{E}) 
\leq n \log n + \pi - \log\left(1 - \tfrac{1}{2\pi}\right).
\end{split}
\]
\end{theorem}

From this, we immediately obtain the asymptotic estimate
\[
\widehat{h}^0(\overline{E}) - \widehat{h}^1(\overline{E}) - \widehat{\deg}(\overline{E}) = O(n \log n),
\]

where the error term $O(n \log n)$ depends only on the rank $n$ of the lattice $E$. In view of \eqref{ARR}, we may interpret Theorem~\ref{main} as an arithmetic Riemann--Roch theorem for normed $\mathbb{Z}$-modules.\\

\vskip 0.5cm

A particularly important class of Euclidean lattices consists of those lattices $\overline{E}$ for which the underlying $\mathbb{Z}$-module $E$ admits a basis that is orthogonal with respect to the scalar product on $\overline{E}$. Such lattices arise naturally in the arithmetic geometry of toric varieties, as we now recall.

Let $\mathcal{X}$ be a smooth toric variety over $\mathrm{Spec}(\mathbb{Z})$, equipped with an action of the torus $\mathbb{T}$. Then $\mathcal{X}$ is determined by a complete nonsingular fan in the real vector space $N \otimes_\mathbb{Z} \mathbb{R} \simeq \mathbb{R}^n$, where $N$ is a free $\mathbb{Z}$-module of rank $n$. Denote by $M = N^\vee$ the dual $\mathbb{Z}$-module, and let $M_\mathbb{R} = M \otimes_\mathbb{Z} \mathbb{R}$. 

Let $\mathcal{L}$ be a $\mathbb{T}$-equivariant line bundle on $\mathcal{X}$ that is generated by its global sections. Then there exists a $\mathbb{T}$-Cartier divisor $D$ on $\mathcal{X}$ such that $\mathcal{L} \simeq \mathcal{O}(D)$. The divisor $D$ determines a rational convex polytope $\Delta_D \subset M_\mathbb{R}$. The space of global sections of $\mathcal{O}(D)$ is then described combinatorially by the lattice points in $\Delta_D$, as follows:
\[
H^0(\mathcal{X}, \mathcal{O}(D)) = \bigoplus_{m \in \Delta_D \cap M} \mathbb{Z} \cdot \chi^m,
\]
where $\chi^m$ denotes the character associated with $m \in M$. For details, we refer the reader to \cite{Fulton, Oda}.

Following \cite[Section 3]{Burgos3}, we define a Hermitian line bundle $\overline{\mathcal{O}(D)} := (\mathcal{O}(D), \|\cdot\|)$ on $\X$ as toric if $D$ is a toric divisor and its associated Green function is invariant under the action of $\mathbb{S}$, the compact torus of $\X(\C)$. Let $\mu$ denote a smooth volume form on $\X(\C)$, which is invariant with respect to the action of $\mathbb{S}$. Furthermore, let $\overline{\mathcal{O}(D)}$ be a toric, continuous Hermitian line bundle on $\X$.  One can see that $(\chi^m)_{m \in \Delta_D \cap M}$ forms a $\mathbb{Z}$-basis of $\overline{H^0(\mathcal{X}, \mathcal{O}(D))}_{L^2, \overline{\mathcal{O}(D)}}$ that is orthogonal with respect to the Euclidean norm $\|\cdot\|_{L^2, \overline{\mathcal{O}(D)}}$. This observation plays a key role in the proof of the integral representation for the arithmetic volume of toric Hermitian line bundles (see \cite[Lemma 2.2]{Moriwaki} and \cite{MounirKJM, Burgos3}). For further background on the Arakelov geometry of toric varieties, we refer the reader to \cite{Burgos3, Burgos2}.\\

The classical theory of Euclidean lattice reduction seeks to construct distinguished bases of Euclidean lattices, commonly referred to as \emph{reduced bases}. Loosely speaking, this theory demonstrates that an Euclidean lattice of rank \( n > 0 \) can be effectively approximated by a direct sum of rank one Euclidean lattices, expressed as \( \overline{E}_1 \oplus \overline{E}_2 \oplus \ldots \oplus \overline{E}_n \), where \( \overline{E}_1, \ldots, \overline{E}_n \) denote rank one  Euclidean lattices. This approximation is achieved with a controlled error that depends on \( n \). Consequently, the lattice can be approximately characterized by \( n \) real parameters \( \mu_i := \widehat{\deg}\, \overline{E}_i \). For detailed presentations and relevant references, the reader is referred to \cite{LLS, LLL}.

Motivated by the above discussion, we  prove the following result.

\begin{theorem}\label{thmOrthog}
Let $\overline{E} = (E, \|\cdot\|)$ be a Euclidean lattice of rank $n$, and suppose that $\overline{E}$ admits an orthogonal $\mathbb{Z}$-basis. Then the following inequality holds:
\[
-\frac{1}{2}n  \log n+\log\left(1-\tfrac{1}{2\pi}\right)\leq \hat h^0(\overline E)+\sum_{i=1}^n \log\min(\lambda_i(\overline E),1)\leq \pi+ n\log \frac{3}{2},
\]
where $\lambda_i(\overline E)$ is the $i$-th successive minimum of $\overline E$, see Section \ref{sm} for the definition.
\end{theorem}

We generalize this result in Corollary~\ref{cor1} by showing that, for any normed $\mathbb{Z}$-module $\overline{E}$ of rank $n$, the estimate

$$
\hat{h}^0(\overline{E}) + \sum_{i=1}^n \log \min(\lambda_i(\overline{E}), 1) = O(n \log n)
$$
holds, where  the error term $O(n \log n)$ depends only on $n$.\\

\
In Paragraph \ref{AB}, we study the notion of arithmetic bigness in Arakelov geometry. The results presented in this section are primarily applications of the theory developed in this paper. 
Let $\X$ be an arithmetic variety over $\mathbb{Z}$ of dimension $n+1$. We assume that the generic fibre $\X_\Q$ is smooth. Let $\overline{\LL} = (\LL, \|\cdot\|_{\overline{\LL}})$ be a smooth Hermitian line bundle on $\X$.  Yuan \cite{YuanInventiones} introduced the condition:

\[
\underset{k \to \infty}{\liminf} \frac{\log \# \{ s \in H^0(\mathcal{X}, k\mathcal{L}) \mid \|s\|_{\sup, k\overline{\LL}} < 1 \}}{k^{n+1}/(n+1)!} > 0,
\]

as a criterion for defining an arithmetically big  line bundle. Moriwaki \cite{Moriwaki2} proposed an alternative definition for arithmetic big line bundles: $\overline{\LL}$ is said to be arithmetically big if $\LL_\Q$ is big and there exists a positive integer $k$ and a nonzero global section $s$ of $k\overline{\LL}$ such that $\|s\|_{\sup, k\overline{\LL}} < 1$. He demonstrated that Yuan's definition is equivalent to the existence of a nonzero section of a power of $\LL$ with supremum norm less than 1, and that $\LL_\Q$ is big \footnote{That is, $\mathrm{vol}(\mathcal{L}_\mathbb{Q}) > 0$, which by definition means

    \[
    \limsup_{k\to\infty} \frac{h^0(\mathcal{X}_\mathbb{Q},\,k\mathcal{L}_\mathbb{Q})}{k^n/n!} > 0,
    \]
    where $n = \dim \mathcal{X}_\mathbb{Q}$.}.

We provide an alternative proof of Moriwaki's result; see Theorem \ref{app} and the discussion following it.

\section{Normed $\Z$-modules }\label{sec2}

A normed $\Z$-module $\overline{E}=(E,\|\cdot\|)$ is  a $\Z$-module of finite type
endowed with a norm $\|\cdot\|$ on the $\C$-vector space
$E_\C=E\otimes_\Z \C$.  Let $E_{\mathrm{tors}}$ denote  the torsion-module 
of $E$, $E_{\mathrm{free}}=E/E_{\mathrm{tors}}$, and 
$E_\R=E\otimes_\Z\R$. We let $B(E,\|\cdot\|)=\{m\in E_\R\mid  \|m\|\leq 1\}$. There exists
a unique Haar measure  on  $E_\R$ such that the volume of $B(E,\|\cdot\|)$ is $1$. We let
\[
\widehat{\chi}(E,\|\cdot\|)=\log \# E_{\mathrm{tors}}-\log \mathrm{vol}(E_\R/ (E/ E_{\mathrm{tors}}) ).
\]

Equivalently, we have
\[
\widehat{\chi}(E,\|\cdot\|)=\log \# E_{\mathrm{tors}}-\log\left( \frac{\mathrm{vol}(E_\R/(E/E_{\mathrm{tor}} ))}{\mathrm{vol}(B(E,\|\cdot\|))}\right),
\]
for any choice of a Haar measure of
$E_\R$. \\

The arithmetic degree of $(E,\|\cdot\|)$ is defined as follows

\[
\widehat{\deg}(E,\|\cdot\|)=\widehat{\deg} \,\overline{E}=  \widehat{\chi}(\overline{E})-\widehat{\chi}(\overline{\Z}^n),
\]
where $\widehat{\chi}(\overline{\Z}^n)=-\log \left( \Gamma(\frac{n}{2}+1)\pi^{-\frac{n}{2}}  \right)$, with $n$ is the rank of $E_\R$.  \\

When the norm $\|\cdot\|$ is induced by a Hermitian  product $\left<\cdot,\cdot\right>$,
 we have 
\[
\widehat{\deg}(\overline{E})=\log \# E/(s_1,\ldots,s_n)-\log \sqrt{\det(\left<s_i,s_j\right>)_{1\leq i,j\leq n}  },
\]
where $s_1,\ldots,s_n$ are elements of $E$ such that their images 
in $E\otimes_\Z\Q$ form a basis. \\

We define $\widehat{H}^0(\overline{E})$ and  $\widehat{h}^0(\overline{E})$ 
to be
\[
\widehat{H}^0(\overline{E})=\left\{ m\in E \mid \|m\|\leq 1  \right\}\quad\text{and}\quad \widehat{h}^0(\overline{E})=\log \# \widehat{H}^0(\overline{E}).
\]
We let
\[
\widehat{H}^1(\overline{E}):=\widehat{H}^0(\overline{E}^\vee)
\quad\text{and}\quad \widehat{h}^1(\overline{E}):=\widehat{h}^0(\overline{E}^\vee),
\]
where $\overline{E}^\vee$ is the $\Z$-module $E^\vee=\mathrm{Hom}_\Z(E,\Z)$ endowed with the dual norm $\|\cdot\|^\vee$ defined as follows
\[
\|f\|^\vee=\sup_{m\in E_\R\setminus\{0\}}\frac{ |f(m)|}{\|m\|}, \quad\forall f\in E^\vee.
\]

Gillet and Soul\'e \cite{Soule-lattice}  proved  an arithmetic analogue of geometric Riemann-Roch theorem for curves. It can be stated as follows: \begin{equation}\label{h0h1}
-\log(6) \ \mathrm{rk}\ E \leq  \widehat{h}^0(\overline{E})-\widehat{h}^1(\overline{E}) -\widehat{\deg}(\overline{E})\leq \log(\tfrac{3}{2}) \mathrm{rk}\, E+2\log ((\mathrm{rk}\ E)!),
\end{equation}
see  \cite[Proposition 2.1]{Moriwaki2} and  also \cite{YuanInventiones}.\\ 

In this paper,  $\overline E_t$ will denote  the normed $\Z$-module  $E$ endowed with the norm 
$t\|\cdot\|$ where  $t>0$.\\

Let $v_n$ denote the volume of the unit ball in $\R^n$ endowed with its standard Euclidean structure.
It is known that
\[
v_{{n}}=\frac{\pi^{n/2} }{\Gamma(\frac{{n}}{2}+1)}.
\]
Stirling's formula gives that
\begin{equation}\label{stirling}
-\frac{n}{2}\log n+n\log 2\leq \log v_n\leq -\frac{n}{2}\log n+ \frac{n}{2}(\log (2\pi)+1)\quad \forall n\gg 1.
\end{equation}

\section{On  theta invariants of  Euclidean lattices}\label{tt}

In this section, we review some elements of the theory of Euclidean lattices. 
Let $\overline E=(E,\|\cdot\|_{\overline E})$ be  a Euclidean lattice over $\Z$ of rank $n$. 
  We recall that this means that we are given the following data: 
A finite-dimensional $\R$-vector space $E_\R$, ($n:=\dim_\R E_\R$),  
 a lattice $E$ in $E_\R$, and $\{e_1,\ldots,e_n\}$  a $\R$-basis of $E_\R$ such that $E=\oplus_{i=1}^n \Z e_i$,  and  a Euclidean norm $\|\cdot\|_{\overline E}$ associated to some Euclidean scalar product $\langle \cdot,\cdot \rangle $ on $E_\R$. 
 
Let  $\lambda_{\overline E}$ be the  unique translation-invariant Radon measure on $E_\R$ which satisfies the following normalization condition:  for every orthonormal basis $\{e_1,\ldots,e_n\}$ of $(E_\R,\|\cdot\|_{\overline E})$,
\[
\lambda_{\overline E} \Bigl(\sum_{i=1}^n [0,1[e_i \Bigr)=1.
\] 
We set
\[
\mathrm{covol}(\overline E):=\mathrm{vol}(E_\R/E),\]
which is, by definition,  $\lambda_{\overline E}(\sum_{i=1}^N [0,1[v_i)$
for every $\Z$-basis $\{v_1,\ldots,v_n\}$ of $E$. $\mathrm{covol}(\overline E)$ is called the covolume of $\overline E$.  Note that \[
\widehat{\deg}(\overline E)=-\log \mathrm{covol}(\overline E).
\]

$\overline E$ induces a natural Euclidean structure on the dual $E^\vee$. We denote the resulting Euclidean lattice by $\overline E^\vee$.\\

We let 
\[
\theta_{\overline E}(t)=\sum_{v\in E} e^{-\pi t  \|v\|_{\overline E}^2}\quad (t>0).
\]
$\theta_{\overline E}$ is called \emph{the theta series} associated  with $\overline E$.

By the Poisson summation formula, we obtain a relation between the theta series of $\overline E$ and $\overline E^\vee$. That is
\begin{equation}\label{poisson}
\sum_{v\in E} e^{-\pi \|v\|_{\overline E}^2}=(\mathrm{covol}(\overline E))^{-1} \sum_{v^\vee \in E^\vee} e^{-\pi \|v^\vee\|_{\overline E^\vee}^2}.
\end{equation}

One can attach to $\overline E$ another arithmetic invariant  $h^0_\theta(\overline E)$ called the \emph{theta invariant} of $\overline E$. It is given as follows:
\[
h_\theta^0(\overline E):=\log \theta_{\overline E}(1).
\]
We let 
\[
h_\theta^1(\overline E):=h_\theta^0(\overline E^\vee).
\]

 The equation \eqref{poisson} may be written 
  as follows:
 \begin{equation}\label{LogPoisson}
h_\theta^0(\overline E)-h_\theta^{1}(\overline E)-\widehat{\deg}(\overline E)=0.
\end{equation}

\begin{proposition}\label{strictineq1}
Let $\overline E$ be an Euclidean lattice. We have 
\begin{equation}\label{ArTheta}
\begin{split}
h_\theta^0(\overline E)-\frac{1}{2}\mathrm{rk} \;E  \log \mathrm{rk}\;E+\log\left(1-\tfrac{1}{2\pi}\right)\leq  & \log \#\{ v \in E | \,  \|v\|_{\overline E}<1\}  \\
& \leq \log \#\{ v \in E | \,  \|v\|_{\overline E}\leq 1\}  \leq h_\theta^0(\overline E)+\pi,
\end{split}
\end{equation}
where \(\text{rk}\,E\) denotes the rank of the lattice \(E\).

\end{proposition}

\begin{proof}  See  \cite{Bana} or \cite{BostTheta}.  For reader's convenience, we recall the proof of
\eqref{ArTheta}.  By  the Poisson summation formula, we have 
\begin{equation}\label{thetaincreases}
\log \theta_{\overline E}(t)+\frac{1}{2}\mathrm{rk}\, E \log t+\log \mathrm{covol }(\overline E)=\log \theta_{\overline E^\vee}(\frac{1}{t})\quad \forall t>0.
\end{equation}
We differentiate this equation to get that
\[
2\pi \sum_{v\in E} \|v\|_{\overline E}^2 \frac{e^{-\pi t \|v\|^2_{\overline E}}}{
\sum_{u\in E} e^{-\pi t \|u\|_{\overline E}^2}}+
\frac{2\pi}{t^2} \sum_{v^\vee \in E^\vee} \|v^\vee\|_{\overline E^\vee}^2 \frac{e^{-\frac{\pi}{ t} \|v^\vee\|^2_{\overline E^\vee}}}{
\sum_{u^\vee\in E^\vee} e^{-\frac{\pi}{ t} \|u^\vee\|_{\overline E^\vee}^2}}=\frac{\mathrm{rk}\, E}{t}\quad \forall t>0.
\]
It follows that
\[
 \sum_{v\in E} \|v\|_{\overline E}^2 e^{-\pi t \|v\|^2_{\overline E}}\leq \frac{\mathrm{rk}\, E}{2\pi t} \sum_{u\in E} e^{-\pi t \|u\|_{\overline E}^2}\quad \forall \,t>0.
\]
From which we infer the following inequality
\begin{equation}\label{10111}
\left(1-\frac{\mathrm{rk}\, E}{2\pi t} \right)\sum_{u\in E} e^{-\pi t \|u\|_{\overline E}^2}\leq \sum_{\substack{u\in E \\
\|u\|_{\overline E} <1}} e^{-\pi t \|u\|_{\overline E}^2}\quad \forall t>0.
\end{equation}
Let $t>\max(1,\frac{\mathrm{rk}\, E}{2\pi})$. We have
\[
\begin{split}
\log \#\{ v \in E | \,  \|v\|_{\overline E}<1\} &\geq  \log\Bigl( \sum_{\substack{v\in E\\
\|v\|<1 }} e^{-\pi t \|v\|^2} \Bigr) \\
&\geq  \log \theta_{\overline E}(t)+\log\left(1-\frac{\mathrm{rk}\, E}{2\pi t}\right) \quad \text{(by \eqref{10111})}\\
&\geq \log \theta_{\overline E}(1)-\frac{\mathrm{rk}\, E}{2} \log t+\log\left(1-\frac{\mathrm{rk}\, E}{2\pi t}\right)\quad 
\text{(by \eqref{thetaincreases})}.\\
\end{split}
\]
By taking $t=\mathrm{rk}\, E$, we obtain
\[
\log \#\left\{ v \in E | \,  \|v\|_{\overline E}<1\right\} \geq 
h_\theta^0(\overline E)-\frac{\mathrm{rk}\, E}{2} 
\log \mathrm{rk}\, E+\log\left(1-\frac{1}{2\pi}\right).
\]
On the other hand, it is clear that
\[
\log \#\left\{ v \in E | \,  \|v\|_{\overline E}\leq 1\right\} \leq h_\theta^0(\overline E)+\pi.
\]
This concludes the proof of the proposition.

\end{proof}

\begin{proposition}\label{weakGS} Let $\overline E$ be an Euclidean lattice. We have
\[
\begin{split}
-\frac{1}{2}\mathrm{rk}(E) \log \mathrm{rk}(E) +\log\left(1-\tfrac{1}{2\pi}\right)-&\pi \leq \widehat{h}^0(\overline E) -\widehat{h}^1(\overline E)-\widehat{\deg}(\overline E) \\
&\leq  \frac{1}{2}\mathrm{rk}(E) \log \mathrm{rk}(E) +\pi -\log\left(1-\tfrac{1}{2\pi}\right).
\end{split}
\]
\end{proposition}
\begin{proof}
We combine \eqref{ArTheta} with  \eqref{LogPoisson} to conclude the proof of the proposition. 

\end{proof}

\begin{proof}[Proof of Theorem \ref{main}]

Let $\overline E=(E,\|\cdot\|)$ be a normed $\Z$-module of rank $n$.  There exists a Euclidean norm 
$\|\cdot \|_J$ on $E$ satisfying the following
\[
\|\cdot\|\leq \|\cdot\|_J\leq {n}^{\frac{1}{2}}\|\cdot\|.
\]
This norm is called John norm, see for instance \cite[Appendix F, 355]{BostTheta}. 
This gives us the following inequalities.
\[
\widehat{\chi}(\overline E_J)\leq \widehat{\chi}(\overline E)\leq 
\frac{n}{2} \log {n}+\widehat{\chi}(\overline E_J),
\]
and
\[
\widehat{h}^0(\overline E_J)\leq \widehat{h}^0(\overline E)\leq 
\widehat{h}^0((\overline{E}_{J})_{n^{-\frac{1}{2}}} ),
\]
and
\[
\widehat{h}^0(\left((\overline{E}_{J})_{{n}^{-\frac{1}{2}} }\right)^\vee)\leq \widehat{h}^0(\overline E^\vee)\leq \widehat{h}^0(\left(\overline E_J\right)^\vee),
\]
where $\overline E_J=(E,\|\cdot\|_J)$.\\ 
Let $\lambda_{\overline E_J}$ denote the  unique Lebesgue measure  on $E_\R$ that gives the volume $1$ to the unit cube in $(E_\R,\|\cdot\|_J)$. Then
\[
\widehat{\chi}(E,\|\cdot\|_J)=\log \mathrm{vol}(B(E,\|\cdot\|_J))+\widehat{\deg}(\overline E_J).
\]
Consequently, we get
\[
\log v_{{n}} +\widehat{\deg}(\overline E_J)\leq \widehat{\chi}(\overline E)\leq  \log v_{{n}} +
({n}/2) \log {n}+\widehat{\deg}(\overline E_J).
\]

So
\[
\begin{split} 
\widehat{h}^0(\overline E_J)-&\widehat{h}^1(\overline E_J)-
\log v_{{n}}-
({n}/2) \log {n}-\widehat{\deg}(\overline E_J)
\leq  \widehat{h}^0(\overline E)-\widehat{h}^1(\overline E)-\widehat{\chi}(\overline E)\\ 
&\leq \widehat{h}^0((\overline{E}_{J})_{n^{-\frac{1}{2}}})-\widehat{h}^1((\overline{E}_{J})_{n^{-\frac{1}{2}}})-\log v_{{n}} -\widehat{\deg}(\overline E_J).
\end{split}
\]
From Proposition \ref{weakGS}, we obtain
\[
\begin{split}
-{n} \log {n} -
\log v_{{n}}+\log\left(1-\tfrac{1}{2\pi}\right)-\pi & \leq  \widehat{h}^0(\overline E)-\widehat{h}^1(\overline E)-\widehat{\chi}(\overline E)  \\ 
&\leq    
{n} \log {n} +\pi -\log\left(1-\tfrac{1}{2\pi}\right)-\log v_{{n}}.
\end{split}
\]
We use \eqref{stirling}
to end the proof of the Theorem.

\end{proof}

\section{Successive minima and  Arithmetic bigness}\label{sm}

Minkowski defined $n$ successive minima  of a given convex body.   In the context of normed $\Z$-modules, they are given as follows. Let $\overline E$ be a  normed $\Z$-module of positive rank $n$.
The successive minima $(\lambda_i(\overline E))_{i=1,\ldots,n}$ of $\overline E$ are defined as follows:
\[
\lambda_i(\overline E)=\inf\left\{r >0 \mid \mathrm{rank}_\R (\mathrm{Span}_\R(E\cap \{ m\in 
E_\R\mid \|m\|\leq r\}))\geq i \right\}\quad (i=1,\ldots,n).
\]
Clearly
\[
\lambda_1(\overline E)\leq \ldots\leq \lambda_n(\overline E).
\]

\begin{lemma}\label{lemma91224} Let $\overline \Z$ be the Euclidean lattice $\Z$ endowed with the standard norm on $\R$. We have
\[ 
1\leq \min(1,\sqrt{t})\theta_{\overline \Z}(t)\leq  \frac{3}{2}\quad \forall t>0. \]
\end{lemma}

\begin{proof}
Let  $t\geq 1$, then  $\theta_{\overline \Z}(t)\leq \theta_{\overline \Z}(1)$. On the other hand, let   $t\in (0,1)$.  Since $\theta_{\overline \Z}(t)=\frac{1}{\sqrt{t}}\theta_{\overline \Z}(\frac{1}{t})$. Then $\theta_{\overline \Z}(t) \leq \frac{1}{\sqrt{t}}\theta_\Z(1).$
We infer that
\begin{equation}\label{thetaleq}
\theta_{\overline \Z}(t)\leq \frac{\theta_{\overline \Z}(1)}{\min(1,\sqrt{t})}\quad \forall \, t>0.
\end{equation}
Since $\min(1,\sqrt{t})\theta_{\overline \Z}(t)=\min(1,\frac{1}{\sqrt{t}}) \theta_{\overline \Z}(\frac{1}{t})$ for every $t>0$ and  $\theta_{\overline \Z}$ is a nondecreasing function and $\theta_\Z(t)\geq 1$, we obtain
\[
1\leq \min(1,\sqrt{t})\theta_{\overline \Z}(t)
\quad \forall t>0.
\]
Using the geometric growth of the exponential terms, we estimate:
\[\theta_{\overline \Z}(1)=1+2\sum_{n=1}^\infty e^{-\pi n^2}\leq \frac{1+e^{\pi}}{1-e^{-\pi}}\leq \frac{3}{2}.\] 
Thus, we establish the inequality:
\[ 
1\leq \min(1,\sqrt{t})\theta_{\overline \Z}(t)\leq  \frac{3}{2}\quad \forall t>0. \]
This ends the proof of the lemma.

\end{proof}

\begin{proof}[Proof of Theorem \ref{thmOrthog}]
Let  $\{e_1,\ldots,e_n\}$ be an orthogonal  $\Z$-basis of $\overline E$. Without loss of generality, assume that $\|e_1\|\leq \ldots\leq \|e_n\|$. Then
\[
\lambda_i(\overline E)=\|e_i\| \quad (i=1,\ldots,n).
\]
We have
\[
\begin{split}
\theta_{\overline E}(1) \prod_{i=1}^n \min(\lambda_i(\overline E),1)=&\prod_{i=1}^n \theta_{\overline \Z}(\|e_i\|^2)  \prod_{i=1}^n \min(\|e_i\|,1)\\
&\leq \prod_{i=1}^n \frac{ \frac{3}{2}}{\min(\|e_i\|,1)}
  \prod_{i=1}^n \min(\|e_i\|,1)\quad \text{(by Lemma \ref{lemma91224})}\\
  &\leq  \left(\frac{3}{2}\right)^n.
\end{split}
\]
On the other hand, by applying Lemma \ref{lemma91224} once again, we obtain:
\[
\begin{split}
1 \leq &
\prod_{\substack{i=1,\ldots,n \\ \lambda_i(\overline{E}) \leq 1}} \theta_{\overline{\Z}}(\lambda_i(\overline{E})^2) \prod_{\substack{i=1,\ldots,n \\ \lambda_i(\overline{E}) \leq 1}} \lambda_i(\overline{E}) \\
=&
\prod_{\|e_i\| \leq 1} \theta_{\overline{\Z}}(\|e_i\|^2) \prod_{\|e_i\| \leq 1} \|e_i\| \\
&\leq \prod_{i=1}^n \theta_{\overline{\Z}}(\|e_i\|^2) \prod_{i=1}^n \min(\|e_i\|,1)  \\ 
& = \theta_{\overline{E}}(1) \prod_{i=1}^n \min(\lambda_i(\overline E), 1),
\end{split}
\]
where the final inequality follows from the orthogonality property.
\\

Combining these results, we obtain:
\[0\leq \hat h_\theta^0(\overline E)+\sum_{i=1}^n \log\min(\lambda_i(\overline E),1)\leq n\log \frac{3}{2}. \]

Finally, by Proposition \ref{strictineq1}, we conclude:
\[
-\frac{1}{2}n  \log n+\log\left(1-\tfrac{1}{2\pi}\right) \leq \hat h^0(\overline E)+ 
\sum_i \log\min(\lambda_i(\overline E),1)\leq \pi +n \log \frac{3}{2}.
\]
This completes the proof.

\end{proof}

 \begin{theorem}\label{thm181224} Let $\overline E$ be a Euclidean lattice. We have
 \[
 \begin{split}
-\pi  -\log n! +&\log \left( \frac{2}{ e\pi}\right) \frac{n}{2} -\frac{n}{2}\log n+ 
\log\left(1-\frac{1}{2\pi}\right)
 \leq \\ 
 & \hat h^0(\overline E)  +\sum_{i=1}^n\log  \min(\lambda_i(\overline E),1)\leq \pi-
\log\left(1-\frac{1}{2\pi}\right) +n\log n,
 \end{split}
 \]
 where $n$ is the rank of $E$.
 \end{theorem}
 
 \begin{proof}

Let \(E_0\) be the \(\mathbb{Z}\)-submodule of \(E\) generated by the elements of the set \(E \cap \{m \in E_\mathbb{R} \mid \|m\| < 1\}\). Let \(\overline{E}_0\) denote the Euclidean lattice \(E_0\) equipped with the induced Euclidean norm from \(\overline{E}\). It is clear that:
\[
\widehat{h}^1(\overline{E}_0) = 0.
\]

Now, consider the case where \(\lambda_1(\overline{E}) > 1\). In this situation, we observe that:
\[
\widehat{h}^0(\overline{E}) = 0.
\]
Thus, the theorem holds true in this case.\\

Let us consider the case where \(\lambda_1(\overline{E}) < 1\). In this case, the submodule \(E_0\) has positive rank. It is straightforward to see that \(\prod_i \lambda_i(\overline{E}_0) = \prod_i \min(\lambda_i(\overline{E}), 1)\). From equation \eqref{ArTheta}, we obtain the bounds:
\[
e^{-\pi} \leq \theta_{\overline{E}_0^\vee}(1) \leq \frac{2\pi}{2\pi - 1} n_0^{n_0/2},
\]
where \(n_0\) is the rank of \(E_0\).\\

By Minkowski's theorem on successive minima (\cite[Theorem 1, p. 59, Theorem 2, p. 62]{Gruber})
\[
\frac{2^{n_0}}{n_0!} \leq \lambda_1(\overline E_0)\cdots \lambda_{n_0}(\overline E_0) \, \mathrm{vol}(B(E_0,\|\cdot\|))\leq 2^{n_0}.
\]
 where  
 $\mathrm{vol}(\cdot)$ is the volume function with respect to the Lebesgue measure that gives  volume 
 $1$ to $E_\R/E$.  Note that  $
\theta_{\overline E_0}(1) \prod_i \lambda_i(\overline E_0)=\frac{
\prod_i \lambda_i(\overline E_0)}{\mathrm{covol}(\overline E_0)} 
\theta_{\overline E_0^\vee}(1).$

We conclude that
 \[
e^{-\pi} \frac{2^{n_0}}{n_0!v_{n_0} }   \leq \theta_{\overline E_0}(1) \prod_i \lambda_i(\overline E_0)\leq 
 \frac{2\pi }{2\pi -1} n_0^{n_0/2} \frac{2^{n_0}}{v_{n_0}}.
\]
We use \eqref{stirling} to obtain that
  \[
e^{-\pi } \frac{1}{n!}\left( \frac{2}{ e\pi}\right)^{n/2} 
\leq 
 \theta_{\overline E_0}(1) \prod_i \lambda_i(\overline E_0)\leq  \frac{2\pi }{2\pi -1} n^{n}.
\]
Note  that $2^n/n!$ is a decreasing function. \\

Since   $\{m\in E_\R\mid \|m\|<1\}\cap E=B_{\overline E_0}(0,1)\cap E_0$, we can  use  Proposition \ref{strictineq1} to deduce that
 \begin{equation}\label{eq281224}
e^{-\pi } \frac{1}{n!}\left( \frac{2}{ e\pi}\right)^{n/2} n^{-n/2} (1-\frac{1}{2\pi}) \leq  \# (B_{\overline E}(0,1)\cap E) \prod_i \min(\lambda_i(\overline E),1)\leq e^{\pi } \frac{2\pi }{2\pi -1}n^{n}. 
 \end{equation}

It remains to consider the case  when $\lambda_1(\overline E)=1$. We see that \eqref{eq281224} holds for $\overline E_t$ with   $t\in (0,1)$. By letting 
$t\rightarrow 1$, we conclude that \eqref{eq281224} holds for $\overline E$. This ends the proof of the theorem.

     \end{proof}

 \begin{corollary}\label{cor1}
 
 Let $\overline E$ be a normed lattice of rank $n$. We have 
  \[
  \begin{split}
  -\pi  -\log n! +\frac{n}{2} & \log \left( \frac{2}{ e\pi}\right)  - n\log n + 
\log\left(1-\frac{1}{2\pi}\right)\\
& \leq 
  \hat h^0(\overline E)+\sum_{i=1}^n \log \min(\lambda_i(\overline E),1)\leq
 2 \pi -2\log\left(1-\frac{1}{2\pi}\right)+\frac{3n}{2}\log n.
 \end{split}
  \]

 \end{corollary}
 
 \begin{proof}
  Let $\overline E$ be a normed lattice. Let $\|\cdot\|_J$ be the  John norm on $E_\R$ that satisfies 
 \[
\|\cdot\|\leq \|\cdot\|_J\leq {n}^{\frac{1}{2}}\|\cdot\|.
\]
Let us denote by $\overline E_J$ the Euclidean lattice $E$ endowed with
$\|\cdot\|_J$.

We have 
  \[
  \begin{split}  \# (B(E,\|\cdot\|)\cap E) \prod_i & \min(\lambda_i (\overline E),1)
\leq   \# (B (E, \tfrac{1}{\sqrt{n} }\|\cdot\|_J )\cap E) \prod_i \min(\lambda_i(\overline E_{J}),1)   \\
  \leq & e^{\pi} \theta_{(\overline{E}_{J})_{ \frac{1}{\sqrt{n}} }}(1)  \prod_i \min(\lambda_i(\overline E_{J}),1)\\
  \leq & e^{\pi} n^{\frac{n}{2}} \theta_{\overline E_{J}}(1)  \prod_i \min(\lambda_i(\overline E_{J}),1)\\
  \leq & n^{n/2} e^{\pi} \frac{2\pi}{2\pi -1} \# (B(E,\|\cdot\|_J)\cap E )  \prod_i \min(\lambda_i(\overline E_{J}),1)  \\
  \leq & e^{2\pi} \left(\frac{2\pi}{2\pi-1} \right)^2 n^{\frac{3n}{2}},
 \end{split}
 \]
 where we have used Theorem \ref{thm181224} and that  $t\mapsto \log \theta_{\overline E}(t)+\frac{n}{2}\log t$ is an increasing function.\\
 
Note that $\min(\lambda_i(\overline E),1)\geq \frac{1}{\sqrt{n}} \min(\lambda_i(\overline E_J),1)$ for every $i=1,\ldots,n$. We deduce that
\[
 \hat h^0(\overline E_J)+\sum_i \log \min(\lambda_i(\overline E),1) -\frac{n}{2}\log n \leq \hat h^0(\overline E)+\sum_i \log \min(\lambda_i(\overline E),1). 
\]
Using Theorem \ref{thm181224} once again, we derive the following inequality:
\[
-\pi - \log n! + \frac{n}{2} \log \left( \frac{2}{e\pi} \right) - n\log n + 
\log \left( 1 - \frac{1}{2\pi} \right) \leq \hat{h}^0(\overline{E}) + \sum_i \log \min(\lambda_i(\overline{E}), 1).
\]

This concludes the proof of the corollary.
 
 \end{proof}

\subsection{Arihtmetic bigness}\label{AB}
Let $\X$ be an arithmetic variety over $\Z$ of dimension $n+1$ and such that
$\X_\Q$ is smooth. Let
$\overline{\LL}{}=(\LL,\|\cdot\|_{\overline\LL})$ be a  smooth Hermitian  line bundle on $\X$.
  For any $k\in \N$, we write 
  $k\overline \LL:={\overline \LL}^{\otimes k}$,  we  let  $n_k$ denote    the rank of $H^0(\X,k\LL)$.
We set $X:=\X(\C)$, and $L:=\mathcal{L}(\C)$.   Let 
$\mu$ be a smooth volume form on $\LL$. The space of global sections $H^0(X,L) $ is endowed with the $L^2$-norm
\[
\|s\|_{L^2,\overline \LL}^2:=\int_X \|s(x)\|_{\overline \LL}^2\mu\quad \text{for any $s\in H^0(X,L)$.}
\]

Also we consider the sup norm defined as follows
\[
\|s\|_{\sup,  \overline \LL}:=\sup_{x\in X}\|s(x)\|_{\overline \LL}\quad \text{for any $s\in H^0(X,L)$.}
\]
For an introduction to Arakelov geometry, see \cite{Soule}.\\

There are several notions of arithmetic positivity for a Hermitian line bundle on an arithmetic variety. We refer the reader to \cite{Moriwaki2, ZhangPositive}, or to \cite[p.~227]{Burgos3} for a detailed discussion of these concepts.

A Hermitian line bundle $\overline{\mathcal{L}}$ on $\mathcal{X}$ is said to be \emph{big} if:
\begin{itemize}
    \item The generic fiber $\mathcal{L}_\mathbb{Q}$ is big,\footnote{That is, $\mathrm{vol}(\mathcal{L}_\mathbb{Q}) > 0$, which by definition means

    \[
    \limsup_{k\to\infty} \frac{h^0(\mathcal{X}_\mathbb{Q},\,k\mathcal{L}_\mathbb{Q})}{k^n/n!} > 0,
    \]
    where $n = \dim \mathcal{X}_\mathbb{Q}$.}
    \item There exists a positive integer $k$ and a nonzero section $s \in H^0(\mathcal{X},\,k\mathcal{L})$ such that $\|s\|_{\sup,\,k\overline{\mathcal{L}}} < 1$.
\end{itemize}

A Hermitian line bundle $\overline{ \mathcal A}$ is said to be ample if: 
\begin{itemize}

    \item $\mathcal{L}$ is ample on $\mathcal{X}$,

    \item The first Chern form $c_1(\overline{\mathcal{L}})$ is positive on $\mathcal{X}(\mathbb{C})$, and
    \item For a sufficiently large integer $k$, the space $H^0(\mathcal{X}, k\mathcal{L})$ is generated by the set

\[
    \{s \in H^0(\mathcal{X}, k\mathcal{L}) \mid \|s\|_{\sup, k\overline{\LL}} < 1\},\]
    as a $\mathbb{Z}$-module.
\end{itemize}

Following the convention in \cite[Convention~9, p.~411]{Moriwaki2}, we write \(\overline{\LL} \leq \overline{\mathcal M}\) if there is an injective homomorphism \(\phi: \LL \rightarrow \mathcal M\) such that \(\|\phi_{\mathbb{C}}(\cdot)\|_{\mathcal M} \leq \|\cdot\|_{\LL}\) on \(\mathcal{X}(\mathbb{C})\), where \(\|\cdot\|_{\LL}\) and \(\|\cdot\|_{\mathcal M}\) are the Hermitian norms associated to \(\overline{\LL}\) and \(\overline{\mathcal M}\), respectively.

\begin{lemma}\label{lm4.4}
Let $\overline{\LL} = (\LL, \|\cdot\|)$ be a big Hermitian line bundle on $\mathcal{X}$. Then there exists a positive integer $\ell_0$ such that

\[
\widehat{h}^0\left(\overline{H^0(\mathcal{X}, \ell \LL)}_{\sup,\, \ell\overline\LL}\right)\neq 0
\quad \text{for all } \ell\geq \ell_0.
\]
\end{lemma}

\begin{proof}
Since $\overline{\LL}$ is big, there exist $k_0 \in \mathbb{N}$ and a nonzero section $s \in H^0(\mathcal{X}, k_0\LL)$ such that $\|s\|_{\sup,\, k_0\overline{\LL}} < \alpha$ for some real number $0<\alpha<1$.

It is known that the sequence

\[
\left(\frac{1}{k}\log \lambda_1\left(\overline{H^0(\mathcal{X}, k\LL)}_{\sup,\, k\overline{\LL}}\right)\right)_{k\in\mathbb{N}}
\]
converges to a finite limit as $k \to \infty$. In particular,

\[
\lim_{k\rightarrow \infty}\frac{1}{k}\log \lambda_1\left(\overline{H^0(\mathcal{X}, k\LL)}_{\sup,\, k\overline{\LL}}\right)
=
\lim_{k\rightarrow \infty}\frac{1}{k k_0}\log \lambda_1\left(\overline{H^0(\mathcal{X}, k k_0\LL)}_{\sup,\, k k_0\overline{\LL}}\right)
\leq \frac{1}{k_0}\log \alpha < 0.
\]
Let $0<\varepsilon<-\frac{1}{k_0}\log \alpha$. By the convergence above, there exists $\ell_0\in \mathbb{N}$ such that for all $\ell\geq \ell_0$,

\[
\frac{1}{\ell}\log \lambda_1\left(\overline{H^0(\mathcal{X}, \ell \LL)}_{\sup,\, \ell\overline{\LL}}\right)
\leq \lim_{k\to\infty}\frac{1}{k}\log \lambda_1(\overline{H^0(\mathcal{X}, k\LL)}_{\sup, k\overline{\LL}})+ \varepsilon
\leq  \frac{1}{k_0}\log \alpha + \varepsilon < 0.
\]
Thus, for all $\ell\geq \ell_0$, we have $\log \lambda_1\left(\overline{H^0(\mathcal{X}, \ell\LL)}_{\sup,\, \ell\overline{\LL}}\right)<0$, i.e.,

\[
\lambda_1\left(\overline{H^0(\mathcal{X}, \ell\LL)}_{\sup,\, \ell\overline{\LL}}\right)<1.
\]
So $\widehat{h}^0(\overline{H^0(\mathcal{X}, \ell\LL)}_{\sup, \ell\overline{\LL}})\neq 0$ for all $\ell\geq \ell_0$.
\end{proof}

\begin{proposition}\label{prop4.5.5}
Let $\overline{\mathcal{A}}$ be an ample Hermitian line bundle on $\mathcal{X}$. Then,

\[
 \liminf_{k \to \infty} \frac{\widehat{h}^0(\overline{H^0(\mathcal{X}, k\A)}_{\sup, k\overline{\A}})}{k^{n+1}/(n+1)!}>0.
\]

\end{proposition}

\begin{proof}
Since $\overline{\mathcal{A}}$ is ample, there exists an integer $k_0 > 0$ such that the graded algebra $\bigoplus_{m \in \N} H^0(\mathcal{X}, m k_0 \mathcal{A})$ is generated by the set
\[
S := \left\{ s \in H^0(\mathcal{X}, k_0 \mathcal{A}) \;\middle|\; \|s\|_{\sup,\, k_0 \overline{\mathcal{A}}} < 1 \right\}.
\]
Define
\[
\varepsilon := -\sup_{s \in S} \log \|s\|_{\sup,\, k_0 \overline{\mathcal{A}}} > 0.
\]
By construction, for each $k \geq 1$, we can find a basis of $H^0(\mathcal{X}, k k_0 \mathcal{A})$ consisting of sections whose sup-norm is at most $e^{-\varepsilon k}$. Consequently, the $n_{k k_0}$-th successive minimum $\lambda_{n_{k k_0}}(k k_0 \overline{\mathcal{A}})$ satisfies

\[
\lambda_{n_{k k_0}}(k k_0 \overline{\mathcal{A}}) \leq e^{-\varepsilon k},
\]
where $n_k$ is the rank of $H^0(\mathcal{X}, k\mathcal{A})$.

By Corollary~\ref{cor1}, we get the following inequality:
\[
O(n_k \log n_k)+ \eps n_{kk_0} k  \leq \hat h^0\left((\overline{H^0(\mathcal{X}, k k_0 \mathcal{A})}_{\sup,\, k k_0 \overline{\mathcal{A}}}\right).
\]
Consequently, 
\[
\varepsilon(n+1)\frac{\mathrm{vol}(\A_\Q)}{k_0} \leq \liminf_{k \to \infty} \frac{\widehat{h}^0\left(\overline{H^0(\mathcal{X}, k k_0 \mathcal{A})}_{\sup,\, k k_0 \overline{\mathcal{A}}}\right)}{(k k_0)^{n+1}/(n+1)!},
\]
noting that $\mathcal{A}_{\mathbb{Q}}$ is big.

Applying Lemma~\ref{lm4.4} to \(\overline{\mathcal{A}}\), we obtain a positive integer \(\ell_0\) such that, for every \(\ell \geq \ell_0\), 
\[
\widehat{h}^0\left(\overline{H^0(\mathcal{X}, \ell \mathcal{A})}_{\sup,\, \ell \overline{\mathcal{A}}}\right)\neq 0.
\]
 Given $m \geq k_1 k_0 + k_0 + \ell_0$, write $
m = k k_0 + r + \ell_0,$ 
where $k \in \mathbb{N}$ and $r \in \{0, \ldots, k_0-1\}$. By Lemma~\ref{lm4.4}, we have
\[
k k_0 \overline{\mathcal{A}} \leq k k_0 \overline{\mathcal{A}} + r \overline{\mathcal{A}} + \ell_0 \overline{\mathcal{A}}.
\]

Therefore,
\[
\begin{split}
\frac{\widehat{h}^0\left(\overline{H^0(\mathcal{X}, m\mathcal{A})}_{\sup,\, m \overline{\mathcal{A}}}\right)}{m^{n+1}}
&\geq \frac{\widehat{h}^0\left(\overline{H^0(\mathcal{X}, k k_0\mathcal{A})}_{\sup,\, k k_0 \overline{\mathcal{A}}}\right)}{(k k_0)^{n+1}} \cdot \frac{(k k_0)^{n+1}}{(k k_0 + r + \ell_0)^{n+1}} \\
&\geq \frac{\widehat{h}^0\left(\overline{H^0(\mathcal{X}, k k_0\mathcal{A})}_{\sup,\, k k_0 \overline{\mathcal{A}}}\right)}{(k k_0)^{n+1}} \cdot \frac{(k k_0)^{n+1}}{((k+1)k_0 + \ell_0)^{n+1}}.
\end{split}
\]
It follows that

\[
\liminf_{m \to \infty} \frac{\widehat{h}^0(\mathcal{X}, m \overline{\mathcal{A}})}{m^{n+1}} > 0.
\]
This completes the proof.
\end{proof}

Moriwaki in \cite{Moriwaki1} introduced the \textit{arithmetic volume}  $\widehat{\mathrm{vol}}(\overline{{\LL}})$
 for  
 a Hermitian line bundle  $\overline{{\LL}}$  
  on arithmetic variety $\X$
    which is   an analogue of the geometric volume function. It is given as follows:
\[
\widehat{\mathrm{vol}}(\overline{{\mathcal{L}}}{})=
\underset{k\rightarrow \infty}{\limsup} \frac{\hat{h}^0(\overline{H^0(\mathcal{X},k \mathcal{L})}_{\sup,k\overline \LL}  
) 
}{k^{n+1}/(n+1)!} .
\]

Yuan \cite{YuanInventiones} employs the condition
 \[
 \underset{k\rightarrow \infty}{\liminf} \frac{\log \# \{ s\in H^0(\mathcal{X},k \mathcal{L}) \mid  \|s\|_{\sup, {k\overline \LL}}<1 \}  
}{k^{n+1}/(n+1)!}>0,
\]

as a definition of an arithmetic big Hermitian line bundle. Moriwaki \cite{Moriwaki2} proposed an alternative definition for arithmetic big line bundles:  $\overline{\LL}$ is said to be arithmetically big if $\LL_\Q$ is big and there exists a positive integer $k$ and a nonzero global section $s$ of $k\overline{\LL}$ such that $\|s\|_{\sup, k\overline{\LL}} < 1$.  He showed that Yuan's definition is equivalent to the existence of a nonzero section of a sufficiently high tensor power of $\LL$ with sup-norm less than $1$, and that $\LL_\Q$ is big.\\

The following theorem is an application of the theory developed in this paper.

\begin{theorem}\label{app}

We keep the same notations as in the beginning of this section. We have 
\begin{enumerate}
\item
{\small{
\[
 \underset{k\rightarrow \infty}{\limsup} \frac{\log \# \{ s\in H^0(\mathcal{X},k \mathcal{L}) \mid  \|s\|_{\sup, {k\overline \LL}}<1 \}  
}{k^{n+1}/(n+1)!}= \underset{k\rightarrow \infty}{\limsup} \frac{\log \# \{ s\in H^0(\mathcal{X},k \mathcal{L}) \mid  \|s\|_{\sup, {k\overline \LL}}\leq 1 \}  
}{k^{n+1}/(n+1)!}.
\]
}
}

\item
{\small{
\[
 \underset{k\rightarrow \infty}{\liminf} \frac{\log \# \{ s\in H^0(\mathcal{X},k \mathcal{L}) \mid  \|s\|_{\sup, {k\overline \LL}}<1 \}  
}{k^{n+1}/(n+1)!}= \underset{k\rightarrow \infty}{\liminf} \frac{\log \# \{ s\in H^0(\mathcal{X},k \mathcal{L}) \mid  \|s\|_{\sup, {k\overline \LL}}\leq 1 \}  
}{k^{n+1}/(n+1)!}.
\]
}}

\end{enumerate}

\end{theorem}

\begin{proof}
Let us prove (1). (2) can be proved in a similar way.  Let $k$ be a positive integer. We denote by 
$\|\cdot\|_{J_k}$ the  John norm on $H^0(\X,k\LL)_\R$ satisfying 
\[
n_k^{-\frac{1}{2}  } \|\cdot\|_{J_k} \leq \|\cdot\|_{\sup,k\overline{\LL}} \leq \|\cdot\|_{J_k},
\]
where $n_k$ is the rank of $H^0(\X,k\LL)$.  So
{\small{
\[
\begin{split}
\underset{k\rightarrow \infty}{\limsup} \frac{\log \# \{ s\in H^0(\mathcal{X},k \mathcal{L}) \mid  \|s\|_{J_k }\leq 1 \}  
}{k^{n+1}/(n+1)!}  & \leq
\underset{k\rightarrow \infty}{\limsup} \frac{\log \# \{ s\in H^0(\mathcal{X},k \mathcal{L}) \mid  \|s\|_{\sup, {k\overline \LL}}\leq 1 \}  
}{k^{n+1}/(n+1)!}\\
&\leq \underset{k\rightarrow \infty}{\limsup} \frac{\log \# \{ s\in H^0(\mathcal{X},k \mathcal{L}) \mid  n_k^{-\frac{1}{2}}\|s\|_{J_k}\leq 1 \}  
}{k^{n+1}/(n+1)!}.
\end{split}
\]
}
}
Similarly, we get
{\small{
\[
\begin{split}
\underset{k\rightarrow \infty}{\limsup} \frac{\log \# \{ s\in H^0(\mathcal{X},k \mathcal{L}) \mid  \|s\|_{J_k }<1 \}  
}{k^{n+1}/(n+1)!}  & \leq
\underset{k\rightarrow \infty}{\limsup} \frac{\log \# \{ s\in H^0(\mathcal{X},k \mathcal{L}) \mid  \|s\|_{\sup, {k\overline \LL}}<1 \}  
}{k^{n+1}/(n+1)!}\\
&\leq \underset{k\rightarrow \infty}{\limsup} \frac{\log \# \{ s\in H^0(\mathcal{X},k \mathcal{L}) \mid  n_k^{-\frac{1}{2}}\|s\|_{J_k}<1 \}  
}{k^{n+1}/(n+1)!}.
\end{split}
\]
}}
On the other hand, we have
\[
h_\theta^0(\overline{H^0(\X,k\LL)}_{J_k})\leq h_\theta^0({\overline{(H^0(\X,k\LL)}_{J_k})}_{n_k^{-\frac{1}{2}} })\leq h_\theta^0(\overline{H^0(\X,k\LL)}_{J_k})-\frac{n_k}{4}\log n_k,
\]
where we have used the fact that $\log \theta_{\overline E}(t)+\frac{1}{2}\mathrm{rk}\, E\log t$ is a nondecreasing function, see
\eqref{thetaincreases}.  So
\begin{equation}\label{ro1}
\limsup_{k\rightarrow \infty} \frac{h_\theta^0(\overline{H^0(\X,k\LL)}_{J_k})}{k^{n+1}/(n+1)!}=\limsup_{k\rightarrow \infty} \frac{h_\theta^0((\overline{H^0(\X,k\LL)}_{J_k})_{n_k^{-\frac{1}{2}}})}{k^{n+1}/(n+1)!}.
\end{equation}

Combining the inequalities above with \eqref{ro1} and Proposition \ref{strictineq1}, we conclude that
{\small{
\[
 \underset{k\rightarrow \infty}{\limsup} \frac{\log \# \{ s\in H^0(\mathcal{X},k \mathcal{L}) \mid  \|s\|_{\sup, {k\overline \LL}}<1 \}  
}{k^{n+1}/(n+1)!}= \underset{k\rightarrow \infty}{\limsup} \frac{\log \# \{ s\in H^0(\mathcal{X},k \mathcal{L}) \mid  \|s\|_{\sup, {k\overline \LL}}\leq 1 \}  
}{k^{n+1}/(n+1)!}.
\]
}
}

This completes the proof of (1).

\end{proof}

From this theorem, we can deduce that \(\overline{\mathcal{L}}\) is arithmetically big in the sense of Yuan if and only if it is arithmetically big in the sense of Moriwaki. Indeed, let us explain this in detail.

Let \(\overline{\mathcal{L}}\) be a big Hermitian line bundle on \(\mathcal{X}\) in the sense of Moriwaki. Let \(\overline{\mathcal{A}}\) be an ample Hermitian line bundle on $\X$. By the argument in \cite[p.~445]{Moriwaki2}, there exists a positive integer \(p\) such that
\[
p\overline{\mathcal{L}} \geq \overline{\mathcal{A}}.
\]
This implies the following bound:
\[
\liminf_{k \to \infty} \frac{\widehat{h}^0\left(\overline{H^0(\mathcal{X}, pk\mathcal{L})}_{\sup, pk\overline{\mathcal{L}}}\right)}{(pk)^{n+1}}
\geq \frac{1}{p^{n+1}} 
\liminf_{k \to \infty} \frac{\widehat{h}^0\left(\overline{H^0(\mathcal{X}, k\mathcal{A})}_{\sup, k\overline{\mathcal{A}}}\right)}{k^{n+1}}.
\]

By Proposition~\ref{prop4.5.5}, we know that

\[
\liminf_{k \to \infty} \frac{\widehat{h}^0\left(\overline{H^0(\mathcal{X}, k\mathcal{A})}_{\sup, k\overline{\mathcal{A}}}\right)}{k^{n+1}} > 0,
\]
since \(\overline{\mathcal{A}}\) is ample.

Then
\[
\liminf_{k \to \infty} \frac{\widehat{h}^0\left(\overline{H^0(\mathcal{X}, pk\mathcal{L})}_{\sup, pk\overline{\mathcal{L}}}\right)}{(pk)^{n+1}} > 0.
\]
Arguing as in the proof of Proposition \ref{prop4.5.5}, we deduce that
\[
\liminf_{k \to \infty} \frac{\widehat{h}^0\left(\overline{H^0(\mathcal{X}, k\mathcal{L})}_{\sup, k\overline{\mathcal{L}}}\right)}{k^{n+1}} > 0.
\]
Using Theorem \ref{app}, we conclude that $\overline \LL$ is arithmetically big in the sense of Yuan.\\

\vskip 0.5cm

Now, let us suppose that

\[
 \liminf_{k \to \infty} \frac{\log \# \{ s \in H^0(\mathcal{X}, k \mathcal{L}) \mid  \|s\|_{\sup, {k\overline{\mathcal{L}}}} < 1 \}}{k^{n+1}/(n+1)!} > 0.
\]
This assumption implies the existence of a positive constant \( c \) and a positive integer \( k_0 \) such that
\begin{equation}\label{c}
\log \# \{ s \in H^0(\mathcal{X}, k \mathcal{L}) \mid  \|s\|_{\sup, {k\overline{\mathcal{L}}}} < 1 \} \geq c k^{n+1} \quad \text{for all } k \geq k_0.
\end{equation}
Consequently, there exists a nonzero section \( s \in H^0(\mathcal{X}, k_0 \mathcal{L}) \) satisfying
\[
\|s\|_{\sup, k_0 \overline{\mathcal{L}}} < 1.
\]
Next, we aim to show that \( \mathcal{L}_\mathbb{Q} \) is big. According to Corollary \ref{cor1}, we have
\[
\log \# \{ s \in H^0(\mathcal{X}, k \mathcal{L}) \mid  \|s\|_{\sup, {k\overline{\mathcal{L}}}} < 1 \} \leq -n_k \log \lambda_1(\overline{H^0(\mathcal{X}, k\mathcal{L})}_{\sup, k\overline{\mathcal{L}}}) + O(n_k \log n_k).
\]
Combining this with our earlier inequality \eqref{c}, we obtain
\[
ck^{n+1} \leq -k n_k \log \lambda_1(\overline{H^0(\mathcal{X}, k\mathcal{L})}_{\sup, k\overline{\mathcal{L}}})^{\frac{1}{k}} + O(n_k \log n_k).
\]
From this, we can infer that
\[
\frac{c}{-\log \lambda_1(\overline{H^0(\mathcal{X}, k\mathcal{L})}_{\sup, k\overline{\mathcal{L}}})^{\frac{1}{k}}} \leq  \liminf_{k \to \infty} \frac{n_k}{k^n}+O(\frac{\log k}{k}).
\]
Then
\[
0< \liminf_{k \to \infty} \frac{n_k}{k^n}.
\]
This shows that \( \mathcal{L}_\mathbb{Q} \) is indeed big.

\begin{remark}

Note that, as explained in \cite[p.~446]{Moriwaki2}, the proof that $\mathcal{L}_\mathbb{Q}$ is big under the assumption $\widehat{\mathrm{vol}}(\overline{\mathcal{L}}) > 0$ relies on \cite[Theorem~4.4]{Moriwaki2}, which in turn is based on the main technical result of the paper, namely \cite[Theorem~3.1]{Moriwaki2}.

\end{remark}


\bibliographystyle{plain} 

\bibliography{biblio}

\end{document}